\documentclass[11pt]{amsart}

\usepackage{amsmath, amssymb, amsthm}

\usepackage{hyperref}
\usepackage{tikz-cd}

\newtheorem{theorem}{Theorem}[section]
\newtheorem{lemma}[theorem]{Lemma}
\theoremstyle{definition}
\newtheorem{defi}[theorem]{Definition}

\def \Surj {\operatorname{Sur}}
\def \Aut {\operatorname{Aut}}
\def \Hom {\operatorname{Hom}}
\def \Gal {\operatorname{Gal}}
\def \Out {\operatorname{Out}}
\def \Ext {\operatorname{Ext}}
\begin{document}

\author{Will Sawin}
\title[identifying measures by their moments]{Identifying measures on non-abelian groups and modules by their moments via reduction to a local problem}
\address{Department of Mathematics, Columbia University, New York, NY, USA}
\email{sawin@math.columbia.edu}

\maketitle

\begin{abstract} Work on generalizations of the Cohen-Lenstra \cite{CL} and Cohen-Martinet \cite{CM} heuristics has drawn attention to probability measures on the space of isomorphism classes of profinite groups. As is common in probability theory, it would be desirable to know that these measures are determined by their moments, which in this context are the expected number of surjections to a fixed finite group. We show a wide class of measures, including those appearing in a recent paper of Liu, Wood, and Zurieck-Brown \cite{LWZ}, have this property. The method is to work ``locally" with groups that are extensions of a fixed group by a product of finite simple groups. This eventually reduces the problem to the case of powers of a fixed finite simple group, which can be handled by a simple explicit calculation. We can also prove a similar theorem for random modules over an algebra.

\end{abstract} 

\section{Introduction}

The primary application of the methods of this paper is a new result in function field number theory, which builds heavily on prior work in~\cite{LWZ}. Thus, we begin by reviewing some notation from~\cite{LWZ}. Let $\Gamma$ be a finite group. A $\Gamma$-group is a profinite group with a continuous action of $\Gamma$. 

Let $\mathbb F_q$ be a finite field of order $q$ prime to $|\Gamma|$. A totally real $\Gamma$-extension $K/ \mathbb F_q(t)$ is a Galois extension $K/\mathbb F_q(t)$, totally split over $\infty$, together with an isomorphism $\operatorname{Gal}(K/\mathbb F_q(t)) \cong \Gamma$. For such a $K$, define $K^{\#}$ to be the maximal everywhere unramified extension of $K$ that is totally split over $\infty$ and of order relatively prime to $q(q-1)|\Gamma|$. Then $\Gal (K^\#/K)$ is a $\Gamma$-group, with $\Gamma$ acting by conjugation~\cite[Definition 2.1]{LWZ}.

Let $n_K$ be the sum of the degrees of the primes in $\mathbb F_q(t)$ that ramify in $K$, and let $E_\Gamma(d,q) $ be the set of totally real $\Gamma$-extensions $K/ \mathbb F_q(t)$ with $n_K=d$. 

We would like to study the distribution of $\Gal (K^\#/K)$ as a $\Gamma$-group. (This provides a model for the distribution of the Galois groups of the maximal unramified extensions of totally real $\Gamma$-extensions of $\mathbb Q$ - see \cite{LWZ} for more on this.) To do this, following~\cite{LWZ}, we consider quotients of $\Gal (K^\#/K)$ that embed as a subquotient into a product of finite groups from a fixed list. This is the analogue of studying the distribution of the class group by first studying the distribution of its $n$-torsion part for fixed $n$ - it simplifies the structure of the individual groups under consideration and prevents escape of mass.

For $\mathcal C$ a finite set of $\Gamma$-groups, we say a finite $\Gamma$-group is level-$\mathcal C$ if it is a $\Gamma$-invariant quotient of a $\Gamma$-invariant subgroup of a product of $\Gamma$-groups in $\mathcal C$. For $G$ a $\Gamma$-group, we say $G^{\mathcal C}$ is the inverse limit of all level-$\mathcal C$ quotients of $G$. 

Let $|\mathcal C|$ be the least common multiple of the orders of the elements of $\mathcal C$. 

Our main application calculates the probability that $ \Gal(K^{\#} /K)^{\mathcal C} $ is a given finite level-$\mathcal C$ $\Gamma$-group $H$, in the limit as $q$ goes to $\infty$ first and $d$ goes to $\infty$ second, subject to congruence conditions on $q$. (This is generally the easiest kind of limit studied in number theory over function fields.)
 
 \begin{theorem}\label{main-function-field-theorem} Let $\mathcal C$ be a finite set of $\Gamma$-groups and let $H$ be a finite level-$\mathcal C$ $\Gamma$-group.  Assume $\gcd( |\Gamma|, |\mathcal C|)=1$. Then
 
 \[ \liminf_{n \to \infty} \liminf_{\substack{ q \to \infty \\ \gcd(q, |\Gamma| |\mathcal C|)=1 \\ \gcd(q-1, |\mathcal C|)=1 }} \frac{ \sum_{ d =0}^{n}   \left|\left \{ K \in E_{\Gamma}(d,q)\mid  \Gal(K^{\#} /K)^{\mathcal C} \cong H \right\} \right| } { \sum_{ d=0}^{n} \left|E_{\Gamma}(d,q)\right| } =\]
 
 \[ \limsup_{n \to \infty} \limsup_{\substack{ q \to \infty \\ \gcd(q, |\Gamma| |\mathcal C|)=1 \\ \gcd(q-1, |\mathcal C|)=1 }} \frac{ \sum_{ d =0}^{n}   \left|\left \{ K \in E_{\Gamma}(d,q)\mid  \Gal(K^{\#} /K)^{\mathcal C} \cong H \right\} \right| } { \sum_{ d=0}^{n} \left|E_{\Gamma}(d,q)\right| } =  \mu_{\Gamma} (U_{\mathcal C, H}) \]
 where $\mu_{\Gamma}(U_{\mathcal C,H})$ is defined in~\cite[(3.15)]{LWZ} and given by an explicit formula in~\cite[Theorem 5.14]{LWZ} (taking $u=1$ in both cases).

 \end{theorem}
 
 Notably, it follows from \cite[Theorem 5.12]{LWZ} that \[\sum_{ \substack{ H \textrm{ a level }\mathcal C\textrm{ }\Gamma\textrm{group} }} \mu_{\Gamma} ( U_{\mathcal C,H})=1.\] In other words, there is no escape of mass in this limit. This is the main reason that we needed to consider the level-$\mathcal C$ quotients of the Galois group. (We could equivalently define, as \cite{LWZ} does, a topological space with topology generated by the sets of $G$ such that $G^{\mathcal C}= H$ for all pairs $\mathcal C,H$ and then obtain a statement on convergence of Borel probability measures on this topological space, but since our arguments proceed entirely with level-$\mathcal C$ $\Gamma$-groups, we avoid this.)
 
 This verifies the function field case of~\cite[Conjecture 1.3]{LWZ} with an additional $q \to \infty$ limit and separation into $\limsup$ and $\liminf$. 
 
 To prove this, we use the analogous limiting statement~\cite[Theorem 1.4]{LWZ} for the moments of this distribution, in other words for the sums \[ \frac{ \sum_{ d =0}^{n}   \sum_{ K \in E_{\Gamma}(d,q) } \Surj_\Gamma ( \Gal(K^{\#} /K)^\mathcal C,  H)  } { \sum_{ d=0}^{n} \left|\left \{ E_{\Gamma}(d,q)\right\}\right | } \] where $\Surj_\Gamma$ denotes the number of $\Gamma$-equivariant surjections between two $\Gamma$-groups.
 
 To make this deduction, we need to know that the probability distribution assigning measure $\mu_{\Gamma} (U_{\mathcal C, H})$ to $H$ is determined by its moments. In fact, we prove a statement involving a much more general class of measures on the set of isomorphism classes of finite level-$\mathcal C$ $\Gamma$-groups. (Whenever we discuss a measure on a set, we take that set to have the discrete topology unless otherwise specified.)

 \begin{theorem}[Theorem \ref{main-bound-group}] \label{main-bound-group-intro} Let $\Gamma$ be a finite group. Let $\mathcal C$ be a finite set of finite $\Gamma$-groups. Assume $\gcd(|\Gamma|,|\mathcal C|)=1$. Let $\mu$ be a measure on the set of isomorphism classes of finite level-$\mathcal C$ $\Gamma$-groups. Let $\mu_t$ be a sequence of measures on the same set. Assume that, for each finite level-$\mathcal C$ $\Gamma$-group $H$, we have
\begin{equation*}\label{assumption-limit-intro} \lim_{t \to \infty} \int \Surj_\Gamma (X, H) d\mu_t(X)  = \int \Surj_\Gamma(X, H) d\mu(X) .\end{equation*}
and
\begin{equation*}\label{assumption-bound-intro} \int \Surj_\Gamma(X, H) d\mu(X)  = O (   |H|^{O(1) }) \end{equation*}
Then for every finite level-$\mathcal C$ $\Gamma$-group $H$ we have
\[ \lim_{t \to \infty} \mu_t(H) = \mu(H) .\]

\end{theorem}

This generalizes \cite[Theorem 1.4]{BW}, which proved that a specific measure on $p$-groups is determined by its moments. It also generalizes \cite{WW}, which proved that a measure on abelian $\Gamma$-groups is determined by its moments. (In each case, checking that every measure that agrees with the fixed one on $G^{\mathcal C}$ for all finite sets $\mathcal C$ agrees on the nose is straightforward.) The paper \cite{WW} itself generalizes much earlier work on the abelian case, including \cite{EVW}. 

 It is not possible to prove a version of Theorem~\ref{main-function-field-theorem} with a limit of the form $\lim_{n \to \infty} \lim_{q\to \infty}$ using only the moment statement~\cite[Theorem 1.4]{LWZ}  since if $\mu_1, \mu_2$ are two distinct measures on the set of isomorphism classes of profinite $\Gamma$-groups that happen to have the same moments, then the measure depending on $q$ and $n$ given by $\begin{cases} (1-1/n) \mu_\Gamma + (1/n) \mu_1 & \textrm{if }q \equiv 1\bmod 4 \\ (1-1/n) \mu_\Gamma + (1/n) \mu_2 & \textrm{if }q \equiv 3\bmod 4\end{cases} $ has moments which converge to the moments of $\mu_\Gamma$ in the $\lim_{n\to\infty} \lim_{q\to\infty}$ sense but does not itself converge in the $\lim_{n\to\infty} \lim_{q\to\infty}$ sense since it does not even converge in the $\lim_{q\to\infty}$ sense for any particular value of $q$.  However, the statement with $\limsup$s and $\liminf$s should suffice for most practical purposes.

The idea of the proof of Theorem~\ref{main-bound-group} is, to determine the probability of a random group being a fixed group $H$, we zoom in to a measure supported on extensions of $H$, then zoom in further to extensions of $H$ by products of finite simple groups, then forget the extension data and only look at the product of simple groups. This reduces the measure-from-moments problem to a simpler problem for a measure on products of finite simple groups. Moreover, it suffices to calculate the measure of the group $1$. This we can do by an explicit inclusion-exclusion using $q$-binomial series identities.

This strategy is a more general analogue of that used in~\cite{HB-i} to deduce \cite[Lemma 18]{HB-i} from \cite[Lemma 17]{HB-i}.

We expect a similar strategy, that involves first passing to a measure supported on extensions of $H$, then further to a measure supported on extensions of $H$ by abelian groups, to be able to answer questions raised in \cite[p. 6]{LWZ} on whether the measure $\mu_{\Gamma}$ is supported on groups whose finite index subgroups have finite abelianization.

\vspace{10pt}

We can also prove the analogous statement for finite modules over an algebra. (Jacob Tsimerman alerted me to the importance of this case.) For $R$ a ring and $X, M$ two $R$-modules, let $\Surj_R(X,M)$ denote the number of surjective $R$-module homomorphisms from $X$ to $M$.
 
 \begin{theorem}[Theorem \ref{main-bound-algebra}] \label{main-bound-algebra-intro} Let $R$ be an associative algebra. Assume that there are finitely many isomorphism classes of finite simple $R$-modules, and that $\operatorname{Ext}^1_R$ between two finite $R$-modules is finite. (e.g. $R$ could be finite over $\mathbb Z_p$ for some prime $p$.
 
 Let $\mu$ be a measure on the set of isomorphism classes of finite $R$-modules. Let $\mu_t$ be a sequence of measures on the same set. Assume that, for $M$ a finite $R$-module,
 \begin{equation}\lim_{t \to \infty} \int \Surj_R (X, M) d\mu_t(X)  = \int \Surj_R(X, M) d\mu(X) .\end{equation}
and
\begin{equation} \int \Surj_R(X, M) d\mu(X)  = O (   |M|^{O(1) }) \end{equation}
Then for all finite $R$-modules $M$ we have
\[ \lim_{t \to \infty} \mu_t(M) = \mu(M) .\]

\end{theorem}

 I would like to thank Melanie Wood, Jacob Tsimerman, and the anonymous referee for helpful conversions and comments on a previous draft of this paper. While working on this project, I served as a Clay Research Fellow and, while completing it, was supported by NSF grant DMS-2101491.

\section{ $\Gamma$-groups} 

Given a $\Gamma$-group $H$, we will understand those $\Gamma$ groups $G$ that map to $H$ by understanding the kernels of maps $\pi : G \to H$. These kernels carry extra structure, because $H$ and $\Gamma$ both act on them by outer automorphisms, and we will need the following ad-hoc definition to keep track of the extra structure:

\begin{defi} For a $\Gamma$-group $H$, let $H'= H \rtimes \Gamma$. We say an $[H']$-group  is a group $G$ together with a map from $H'$ to $\operatorname{Out}(G)$.\end{defi}

Note that $\Gamma$ also acts on $\ker \pi$ by honest automorphisms instead of outer automorphisms. We ignore this extra structure to simplify the definition and to simplify certain proofs, as it is not necessary for our arguments.

\begin{defi} We say a homomorphism $G_1\to G_2$ of $[H']$-groups is a $[H']$-homomorphism if for each element of $h\in H'$, for each lift $\sigma_1$ of $h$ from $\Out(G_1)$ to $\Aut(G_1)$, there is a lift $\sigma_2$ of $h$ from $\Out(G_2)$ to $\Aut(G_2)$ such that $\sigma_2 \circ f = f \circ \sigma_1$.

We say $\Surj_{[H']} (G_1, G_2)$ is the number of surjective $[H']$-homomorphisms from $G_1$ to $G_2$. \end{defi}

$\Out(G)$ acts on the set of normal subgroups of $G$, so $H'$ acts on the set of normal subgroups of a $[H']$-group $G$. We say a nontrivial $[H']$-group is simple if there are no nontrivial proper fixed points of this action.

It is easy to see that the composition of two $[H']$-homomorphisms is an $[H']$-homomorphism, so $[H']$-groups form a category.

\vspace{10pt} We will calculate the probability that a random $\Gamma$-group $X$ is isomorphic to $H$ by studying the probability that the kernel of random $\Gamma$-surjection $\pi: X\to H$ is trivial. Thus, our first few lemmas will focus on identifying the probability that an $[H']$-group is trivial from the moments of a distribution on $[H']$-groups. By quotienting by the intersection of all maximal proper $H\rtimes \Gamma$-invariant normal subgroups, we will be able to restrict attention to products of finite simple $[H']$-groups. Thus, our first few lemmas will involve finite simple $[H']$-groups and their products.  

\vspace{10pt}

We begin with a series of lemmas that let us calculate $\Surj_{[H']} (G_1, G_2)$ for products of finite simple $[H']$-groups.

\begin{lemma}\label{abelian-surjection-count} Let $G$ be an abelian finite simple $[H']$-group, and let $h$ be the number of $[H']$-homomorphisms from $G$ to $G$. Then
\begin{equation}\label{eq-abelian-surjection-count} \Surj_{[H']}(G^e,G^k) = (h^{e}-1) (h^e- h) \dots (h^{e}- h^{k-1} ) .\end{equation}
\end{lemma}

\begin{proof} Because $G$ is abelian, its outer automorphism group is equal to its automorphism group, so we may view $G$ as simply an abelian group with an action of $H'$, or as a $\mathbb Z[ H']$-module. In this view, $[H']$-homomorphisms between $G^e$ and $G^k$ are module homomorphisms.

Because $G$ is finite simple as a $[H']$-group, it is a simple module with finite cardinality, so its endomorphism algebra as a module is a finite field $\mathbb F_h$. Then maps from $G^e$ to $G^k$ are $k \times e$ matrices over $\mathbb F_h$, and are surjective if and only if this matrix is surjective. The result \eqref{eq-abelian-surjection-count} then follows because $(h^{e}-1) (h^e- h) \dots (h^{e}- h^{k-1} )$ is the number of $k \times e$ surjective matrices. \end{proof}

\begin{lemma}\label{non-abelian-surjection-count} Let $G$ be a non-abelian finite simple $[H']$-group, and let $| \Aut_{[H']}(G)|$ be the number of $[H']$-bijections from $G$ to $G$. Then
\begin{equation}\label{eq-non-abelian-surjection-count} \Surj_{[H']}{(G^e,G^k)} =  e (e-1) \dots (e+1-k)  | \Aut_{[H']}(G)|^k .\end{equation} 
\end{lemma}

\begin{proof} We first prove that any $[H']$-surjection from $G^e$ to $G$ is the composition of projection onto one factor with an $[H']$-automorphism of $G$. To do this, note when $f: G^e \to G$ is restricted to each copy of $G$, it defines an $[H']$-homomorphism from $G$ to $G$, which because $G$ is $[H']$-simple must be either an $[H']$-automorphism or the trivial map.

These restrictions cannot all be trivial or else $f$ would fail to be surjective, and two cannot be automorphisms or else the images of two different copies of $G$ would each be all of $G$ and would fail to commute with each other, so one is an isomorphism and the rest are trivial. Thus $f$ is the composition of a projection with an $[H']$-automorphism.

Now consider an $[H']$-surjection from $G^e$ to $G^k$. Its composition with each of the $k$ projections remains an $[H']$-surjection, so each of these is a projection onto one factor composed with an automorphism. Conversely, given $k$ such projections onto one of $e$ factors and $k$ automorphisms, the induced map $G^e \to G^k$ is an $[H']$-homomorphism, because the $[H']$-automorphism condition for each of the $k$ factors gives a lift from $\Out(G) $ to $\Aut(G)$, hence a lift from $\Out(G)^k \subseteq \Out(G^k)$ to $\Aut(G)^k \subseteq \Aut(G^k)$.  Furthermore, the induced map $G^e \to G^k$ is surjective if and only if the same projection never appears twice. (We can check this on the level of groups, where it is obvious.)

Because the number of ordered $k$-tuples of choices of one out of $e$ projections, never repeating, is $ e (e-1) \dots (e+1-k) $, and the number of $k$ tuples of $[H']$-automorphisms is $| \Aut_{[H']}(G)|^k$, we obtain \eqref{eq-non-abelian-surjection-count}. 

\end{proof}

\begin{lemma}\label{surjection-count-splitting} Let $G_1,\dots, G_m$ be finite simple $[H']$-groups that are not pairwise $[H']$-isomorphic. Then
\begin{equation}\label{eq-surjection-count-splitting} \Surj_{[H']} \Bigl( \prod_{i=1}^m G_i^{e_i}, \prod_{i=1}^m G_i^{k_i}\Bigr) = \prod_{i=1}^m \Surj_{[H']} \left(G_i^{e_i}, G_i^{k_i} \right) . \end{equation} 
\end{lemma}

\begin{proof} First observe that, given a tuple $f_1,\dots, f_m$ of $[H']$-surjections $f_i: G_i^{e_i}\to G_i^{k_i}$, the product \[ \prod_{i=1}^m f_i\hspace{5pt}
  \colon\hspace{5pt} 
   \prod_{i=1}^m G_i^{e_i}\to  \prod_{i=1}^m G_i^{k_i}\] is an $[H']$-surjection. Certainly $ \prod_{i=1}^m f_i$ is a surjective group homomorphism, so it suffices to check that for any $\sigma_1\in \Aut ( \prod_{i=1}^m G_i^{e_i}) $ lifting $h \in H'$, we can find $\sigma_2 \in \Aut( \prod_{i=1}^m G_i^{k_i} )$ lifting $h$ such that $\sigma_2 \circ \prod_{i=1}^m f_i = \prod_{i=1}^m f_i \circ \sigma_1$. Because $\sigma_1$ lifts an outer automorphism that stabilizes the individual factors $G_i^{e_i}$, $\sigma_1$ stabilizes the factors $G_i^{e_i}$, so in fact $\sigma_1 \in \prod_{i=1}^m \Aut( G_i^{e_i})$.  For each $i$ we can apply the $[H']$-homomorphism property of $f_i$ to the restriction of $\sigma_1$ to $\Aut( G_i^{e_i})$ to find a suitable element of $\Aut(G_i^{k_i})$, and then define $\sigma_2$ to be the product of these elements. This verifies the $[H']$-homomorphism condition.

Thus, it suffices to prove that every $[H']$-surjection $f:  \prod_{i=1}^m G_i^{e_i}\to  \prod_{i=1}^m G_i^{k_i}$ arises this way.

Let us first check that for any $i \neq j$ and any $1 \leq a \leq e_i$, $1 \leq b \leq k_j$, the restriction of $f$ to a map from the $a$th copy of $G_i$ to the $b$th copy of $G_j$ is zero.

Because the $a$th copy of $G_i$ is a normal subgroup of $\prod_{i=1}^m G_i^{e_i}$, and the image of a normal subgroup under any surjection is normal, the image of the restricted map $G_i \to G_j$ is a normal subgroup of $G_j$. 
By the homomorphism condition, this image is also $H'$-invariant, so because $G_j$ is $[H']$-simple the image must be all of $G_j$ or trivial. Similarly, the kernel of this map from $G_i $ to $ G_j$ is $H'$-invariant, so must be $G_i$ or trivial. Thus the map from $G_i$ to $G_j$  either a bijection or zero. It cannot be a bijection as, by assumption, $G_i$ and $G_j$ are not isomorphic, so it must be zero, as desired.

It follows that, as a group homomorphism, $f$ is the product of maps $f_i: G_i^{e_i}\to G_i^{k_i}$. Because $f$ is surjective, these maps $f_i$ are surjective. Because $f$ is an $[H']$-homomorphism, and $f_i$ is the composition of $f$ with the inclusion $G_i^{e_i } \to \prod_{i=1}^m G_i^{e_i}$ and projection $\prod_{i=1}^m G_i^{k_i} \to G_i^{k_i}$, which are both $[H']$-homomorphisms, $f_i$ must be an $[H']$-homomorphism, finishing the proof. 

\end{proof}

 The next two lemmas let us solve a basic version of the problem of reconstructing measures from moments, where we look at a measure on groups which consist of powers of a single group, and reconstruct from the moments only the measure of the trivial group. We define in Lemma \ref{c-k-exists} a sequence of coefficients, and describe its useful properties. In Lemma \ref{one-variable} we use these properties to prove the reconstruction statement, in a uniform way in both the abelian and non-abelian cases.

\begin{lemma}\label{c-k-exists} Let $H'$ be a group and let $G$ be a finite simple $[H']$-group. Then there exist constants $c_k \in \mathbb R$, depending on $G$, such that

\begin{enumerate}

\item $c_0=1$,
\item for any $e>0$ we have \[  \sum_{k=0}^r c_k  \Surj_{[H']}( G^{e}, G^k) \geq 0 \] if $r$ is even and \[  \sum_{k=0}^r c_k   \Surj_{[H']} ( G^{e}, G^k) \leq 0 \] if $r$ is odd, and
\item $|c_k|$ converges to $0$ superexponentially in $k$ (i.e. $|c_k| = O_N ( N^{-k})$ for every $N>0$.)\end{enumerate}  

\end{lemma}

To interpret these conditions, note that (1) and (2) together imply the identity
\[ \sum_{k=0}^{\infty} c_k \Surj_{[H']}( G^{e}, G^k)  = \begin{cases} 1 &\textrm{if }e=0\\ 0&\textrm{otherwise}\end{cases}.\]  Thus, for a measure $\mu$ on the nonnegative integers, we can attempt to reconstruct $\mu(0)$ from the moments $\sum_{e=0}^{\infty}\Surj_{[H']}( G^{e}, G^k) $ by summing the $k$th moment against $c_k$. Using all three conditions, we can prove that this attempt works.

The formula for $c_k$ in the case $G$ abelian is essentially due to Heath-Brown \cite[Equation (22)]{HB-o}, who used this identity and part (3), but not part (2), to prove a weaker uniqueness statement.

\begin{proof} If $G$ is not abelian we have $\Surj_{[H']}( G^{e}, G^k) = e (e-1) \dots (e+1-k)  | \Aut_{[H']}(G)|^k$ by Lemma \ref{non-abelian-surjection-count} and we take \[ c_k =  \frac{ (-1)^k }{ k! | \Aut_{[H']}(G)|^k }.\]  Here (1) is clear, (2) follows from the identity
\[ \sum_{k=0}^{r} (-1)^k  {e\choose k}  = (-1)^r { e-1 \choose r} \] which is $\geq 0$ if $r$ is even and $\leq 0 $ if $r$ is odd, and (3) is clear.

If $G$ is abelian, let $h= \Hom_{[H']}(G, G) $. The number of surjections is given by \eqref{eq-abelian-surjection-count} from Lemma \ref{abelian-surjection-count} and we take $c_k = \frac{(-1)^ k  } { (h-1) \dots (h^k -1 ) }$. Then (1) is clear, (2) follows from the identity
\begin{equation}\label{q-binomial-identity} \sum_{k=0}^r (-1)^k  { e \choose k}_h  h^{ k \choose 2} = (-1)^r { e-1 \choose r }_h h^{ r+1 \choose 2} \end{equation}  which is $\geq 0$ if $r$ is even and $\leq 0$ if $r$ is odd, and (3) is clear. In this identity, the binomial coefficients are interpreted as $q$-binomial coefficients, except with $q=h$. The identity \eqref{q-binomial-identity} follows from the more standard identity 
\[ \binom{e}{k}_h = h^k \binom{e-1}{k}_h + \binom{e-1}{k-1}_h\] by a telescoping sum. \end{proof}

\begin{lemma}\label{one-variable} Let $H$ be a $\Gamma$-group and let $G$ be a finite simple $[H']$-group.  Let $m$ and $m_t$ for $t \in \mathbb N$ be functions from $\mathbb N$ to the nonnegative real numbers. Assume that, for all $k$,
\begin{equation}\label{one-var-bound} \sum_{e=0}^{\infty}  \Surj_{[H']} ( G^e, G^k) m(e) = O \left(  O(1)^k \right) \end{equation}
and
\begin{equation}\label{one-var-limit} \lim_{t\to\infty} \sum_{e=0}^{\infty}  \Surj_{[H']} ( G^e, G^k) m_t (e)= \sum_{e=0}^{\infty}  \Surj_{[H']} ( G^e, G^k) m(e) \end{equation}
Then
\begin{equation}\label{one-var-conc} \lim_{t\to \infty} m_t(0) = m(0) .\end{equation}\end{lemma}

\begin{proof} 
By Lemma \ref{c-k-exists}, part (1) and Lemma \ref{c-k-exists}, part (2), we have
\[ \sum_{k=0}^r c_k \sum_{e=0}^{\infty}  \Surj_{[H']} ( G^e, G^k) m(e) = \sum_{e=0}^{\infty} \left(\sum_{k=0}^rc_k \Surj_{[H']} ( G^e, G^k) \right) m(e) \]\[ = m(0) + \sum_{e=1}^{\infty} \left(\sum_{k=0}^rc_k \Surj_{[H']} ( G^e, G^k) \right) m(e) \]\[  \geq m(0) \] if $r$ is even and \[ \leq m(0)\] if $r$ is odd. Moreover, the analogous inequalities hold for $m_t$ for all $t$.

By Lemma \ref{c-k-exists}(3) and our assumption \eqref{one-var-bound}, $\sum_{k=0}^\infty c_k \sum_{e=0}^{\infty}  \Surj_{[H']} ( G^e, G^k) m(e) $ is a convergent series, so \[ \lim_{r\to \infty} \sum_{k=0}^r c_k \sum_{e=0}^{\infty}  \Surj_{[H']} ( G^e, G^k) m(e) = m(0).\] 

Therefore for all $r$ even we have
\[ \lim\sup_{t\to\infty} m_t(0) \leq \lim\sup_{t\to \infty} \sum_{k=0}^r c_k \sum_{e=0}^{\infty}  \Surj_{[H']} ( G^e, G^k) m_t(e)= \sum_{k=0}^r c_k  \lim_{t\to\infty} \sum_{e=0}^{\infty}  \Surj_{[H']} ( G^e, G^k) m_t(e)  \] \[=  \sum_{k=0}^r c_k \sum_{e=0}^{\infty}  \Surj_{[H']} ( G^e, G^k) m(e)\]
and thus
\begin{equation}\label{one-var-upper} \lim\sup_{t\to\infty} m_t(0) \leq \lim_{ r\to\infty}  \sum_{k=0}^r c_k \sum_{e=0}^{\infty}  \Surj_{[H']} ( G^e, G^k) m(e) = m(0)\end{equation}
and for $r$ odd we have
\[ \lim\inf_{t\to\infty} m_t(0) \geq \lim\inf_{t\to \infty} \sum_{k=0}^r c_k \sum_{e=0}^{\infty}  \Surj_{[H']} ( G^e, G^k) m_t(e)= \sum_{k=0}^r c_k \sum_{e=0}^{\infty} \lim_{t\to\infty}  \Surj_{[H']} ( G^e, G^k) m_t(e)  \] \[=  \sum_{k=0}^r c_k \sum_{e=0}^{\infty}  \Surj_{[H']} ( G^e, G^k) m(e)\]
and thus
\begin{equation}\label{one-var-lower} \lim\inf_{t\to\infty} m_t(0) \geq \lim_{ r\to\infty}  \sum_{k=0}^r c_k \sum_{e=0}^{\infty}  \Surj_{[H']} ( G^e, G^k) m(e) = m(0).\end{equation}
Combining \eqref{one-var-upper} and \eqref{one-var-lower}, we get our desired conclusion \eqref{one-var-conc}.
\end{proof}

The next two lemmas improve the statement of Lemma \ref{one-variable} from powers of a single group to products of powers of a finite list of groups. This is based on an inductive strategy where we handle one group at a time. Lemma \ref{inductive-step} will give the inductive step and Lemma \ref{flat-case} will complete the argument.

\begin{lemma} \label{inductive-step} Let $H$ be a $\Gamma$-group and let $G_1,\dots, G_m$ be finite simple $[H']$-groups. Let $\tilde{\mu}$ be a measure on the set of isomorphism classes of groups $\prod_{i=1}^m G_i ^{e_i}$ and let $\tilde{\mu}_t$ be a sequence of such measures. Let $j$ be a natural number from $1$ to $m$.

Assume that for all $k_j,\dots, k_m \in \mathbb N$ we have
\begin{equation}\label{ind-bound} \sum_{e_{j}, \dots ,e_m=0}^{\infty}  \Surj_{[H']} \left( \prod_{i=j}^m G_i^{e_i} , \prod_{i=j}^m G_i^{k_i}\right)  \tilde{\mu} \left(  \prod_{i=j}^m G_i^{e_i} \right)  = O \left( O(1)^{ \sum_{i=j}^m k_i } \right) \end{equation}
and
\begin{multline}\label{ind-limit}   \lim_{t\to\infty} \sum_{e_{j}, \dots ,e_m=0}^{\infty}  \Surj_{[H']} \left( \prod_{i=j}^m G_i^{e_i} , \prod_{i=j}^m G_i^{k_i}\right)  \tilde{\mu}_t \left(  \prod_{i=j}^m G_i^{e_i} \right)  \\ =   \sum_{e_{j}, \dots ,e_m=0}^{\infty}  \Surj_{[H']} \left( \prod_{i=j}^m G_i^{e_i} , \prod_{i=j}^m G_i^{k_i}\right)  \tilde{\mu} \left(  \prod_{i=j}^m G_i^{e_i} \right). \end{multline}
Then for all $k_{j+1},\dots, k_m \in \mathbb N$ we have
\begin{equation}\label{ind-conc} \begin{split}    \lim_{t \to \infty}  \sum_{e_{j+1}, \dots ,e_m=0}^{\infty}  \Surj_{[H']} \Bigl( \prod_{i=j+1}^m G_i^{e_i} , \prod_{i=j+1}^m G_i^{k_i}\Bigr)  \tilde{\mu}_t \Bigl(  \prod_{i=j+1}^m G_i^{e_i} \Bigr)\\ = \sum_{e_{j+1}, \dots ,e_m=0}^{\infty}  \Surj_{[H']} \Bigl( \prod_{i=j+1}^m G_i^{e_i} , \prod_{i=j+1}^m G_i^{k_i}\Bigr)  \tilde{\mu} \Bigl(  \prod_{i=j+1}^m G_i^{e_i} \Bigr).   \end{split} \end{equation} \end{lemma}

\begin{proof} Fix $k_{j+1},\dots, k_m$.  
Define \[ m(e) =  \sum_{e_{j+1}, \dots ,e_m=0}^{\infty}  \Surj_{[H']} \left( \prod_{i=j+1}^m G_i^{e_i} , \prod_{i=j+1}^m G_i^{k_i}\right)  \tilde{\mu} \left(  G_j^ e \times \prod_{i=j+1}^m G_i^{e_i} \right)\] and \[ m_t(e) =  \sum_{e_{j+1}, \dots ,e_m=0}^{\infty}  \Surj_{[H']} \left( \prod_{i=j+1}^m G_i^{e_i} , \prod_{i=j+1}^m G_i^{k_i}\right)  \tilde{\mu}_t \left(  G_j^ e \times \prod_{i=j+1}^m G_i^{e_i} \right).\]  Thus, for any natural number $k$, defining $k_j=k$, we have by Lemma \ref{surjection-count-splitting},
\begin{equation}\label{mu-m-conversion} \begin{split} & \sum_{e=0}^{\infty} m(e) \Surj_{[H'] } (G_j^e, G_j^k) \\ =& \sum_{e=0}^{\infty}  \sum_{e_{j+1}, \dots ,e_m=0}^{\infty}  \Surj_{[H']} \left( \prod_{i=j+1}^m G_i^{e_i} , \prod_{i=j+1}^m G_i^{k_i}\right)   \Surj_{[H'] } (G_j^e, G_j^k)\tilde{\mu} \left(  G_j^ e \times \prod_{i=j+1}^m G_i^{e_i} \right)\\  = &\sum_{e=0}^{\infty}  \sum_{e_{j+1}, \dots ,e_m=0}^{\infty}  \Surj_{[H']} \left(G_j^e \times  \prod_{i=j+1}^m G_i^{e_i} , G_j^k \times  \prod_{i=j+1}^m G_i^{k_i}\right)  \tilde{\mu} \left(  G_j^ e \times \prod_{i=j+1}^m G_i^{e_i} \right)\\=& \sum_{e_{j}, \dots ,e_m=0}^{\infty}  \Surj_{[H']} \left( \prod_{i=j}^m G_i^{e_i} , \prod_{i=j}^m G_i^{k_i}\right)  \tilde{\mu} \left(  \prod_{i=j}^m G_i^{e_i} \right), \end{split}\end{equation} and a similar identity holds with $m_t$ and $\tilde{\mu}_t$.

We now apply Lemma \ref{one-variable} to $m$ and $m_t$.

Because the number of surjections from any group to another is nonnegative, $m$ and $m_t$ are nonnegative.

Applying \eqref{mu-m-conversion} for $m$ and $m_t$, we see that the hypothesis \eqref{one-var-limit} of Lemma \ref{one-variable} is exactly our assumption \eqref{ind-limit}.

Applying \eqref{mu-m-conversion}, the hypothesis \eqref{one-var-bound} of Lemma \ref{one-variable} follows from our assumption \eqref{ind-bound} since with $k_{j+1},\dots, k_m$ fixed, $O \left( O(1)^{  \sum_{i=j}^m k_i } \right)  = O \left( O(1)^{k_j} \right) =O \left( O(1)^{k} \right)  $.

Because we have verified both hypotheses, we can apply Lemma \ref{one-variable}. We use the definition of $m$ and $m_t$ to see that the conclusion \eqref{one-var-conc} is exactly our desired \eqref{ind-conc}. 

\end{proof}

\begin{lemma}\label{flat-case}
Let $H'$ be a finite group and let $G_1,\dots, G_m$ be finite simple $[H']$-groups. Let $\tilde{\mu}$ be a measure on the set of isomorphism classes of groups $\prod_{i=1}^m G_i ^{e_i}$ and let $\tilde{\mu}_t$ be a sequence of such measures. 

Assume that for all $k_1,\dots, k_m \in \mathbb N$ we have
\begin{equation}\label{flat-bound} \sum_{e_{1}, \dots ,e_m=0}^{\infty}  \Surj_{[H']} \left( \prod_{i=1}^m G_i^{e_i} , \prod_{i=1}^m G_i^{k_i}\right)  \tilde{\mu} \left(  \prod_{i=1}^m G_i^{e_i} \right)  = O \left( O(1)^{ \sum_{i=1}^m k_i } \right) \end{equation}
and
\begin{equation}\label{flat-limit}   \lim_{t \to \infty}  \sum_{e_{1}, \dots ,e_m=0}^{\infty}  \Surj_{[H']} \left( \prod_{i=1}^m G_i^{e_i} , \prod_{i=1}^m G_i^{k_i}\right)  \tilde{\mu}_t \left(  \prod_{i=1}^m G_i^{e_i} \right) =   \sum_{e_{1}, \dots ,e_m=0}^{\infty}  \Surj_{[H']} \left( \prod_{i=1}^m G_i^{e_i} , \prod_{i=1}^m G_i^{k_i}\right)  \tilde{\mu} \left(  \prod_{i=1}^m G_i^{e_i} \right). \end{equation}
Then for all we have
\begin{equation}\label{flat-conc}  \lim_{t \to \infty} \tilde{\mu}_t ( 1 ) = \tilde{\mu} ( 1) . \end{equation}\end{lemma}

\begin{proof} The hypothesis \eqref{ind-limit} and conclusion \eqref{ind-conc} of Lemma \ref{inductive-step} are identical, except that the conclusion has $j+1$ where the hypothesis has $j$. This is exactly what we need for an inductive argument. Based on this idea, we will prove that \eqref{ind-limit} holds for all $j$ from $1$ to $m+1$ by induction on $j$.

The base case of this induction is the $j=1$ case of \eqref{ind-limit}, which is exactly our assumption \eqref{flat-limit}.

The induction step follows from Lemma \ref{inductive-step} once we check the other hypothesis \eqref{ind-bound} of Lemma \ref{inductive-step}. This hypothesis follows from our assumption \eqref{flat-bound} because \eqref{ind-bound} requires us to bound the sum in \eqref{flat-bound}, restricted to the case when $k_1,\dots, k_{j-1}=0$. Because all terms in \eqref{flat-bound} are nonnegative, the restricted sum is bounded by the original sum.

This verifies the induction, and finally we observe that the $j=m+1$ case of \eqref{ind-limit} is our desired conclusion \eqref{flat-conc}.

\end{proof} 

We have now obtained a special case of Theorem \ref{main-bound-group-intro}, where the measure is supported on the products of elements from a specific finite list of groups and we only reconstruct the measure of the identity from the moments, rather than the measure of an arbitrary element. We now reduce the probability that a random group $X$ is $H$ to the probability that a random surjection $X \to H$ is an isomorphism, which we reduce to the probability that the kernel of the random surjection is trivial, which we reduce to the probability that the quotient of the kernel by the intersection of all its maximal proper subgroups is trivial. This strategy was already used in \cite{LWZ}, and less explicitly in other works, to compute a measure on random groups. The key additional observation is that we can calculate the moments of this transformed random model straightforwardly from our original moments, which then allows us to apply Lemma \ref{flat-case}.

Our first lemmas study the quotient of an $[H']$-group by the intersection of its maximal proper $H'$-invariant normal subgroups:

\begin{defi} For $T$ an $[H']$-group, we define $Q(T)$ to be the quotient of $T$ by the intersection of all its maximal proper $H'$-invariant normal subgroups. Because this intersection is $H'$-invariant, $Q(T)$ carries a natural $[H']$-structure. \end{defi}

\begin{lemma}\label{Q-finite-simple} For $T$ a finite $[H']$-group, $Q(T)$ is a product of finite simple $[H']$-groups. \end{lemma}

\begin{proof}Let us check that the quotient by any intersection of $n$ maximal proper $H'$-invariant normal subgroups is a product of finite simple $[H']$-groups, by induction on $n$.  For the induction step, let $Z$ be the intersection of $n$ such subgroups, and let $W$ be another such subgroup. If $W$ contains $Z$,  $W \cap Z=Z$ and we are done. Otherwise, $WZ$ is an $H'$-invariant normal subgroup containing $W$, but not equal to $W$, and so $WZ=T$

Because $WZ=T$, the natural map $T/ (W \cap Z) \to (T/W) \times (T/Z)$ is an isomorphism. Because $W$ and $Z$ are $ H'$-invariant, this map is an $[H']$-homomorphism. By the induction hypothesis, $T/Z$ is a product of finite simple $[H']$-groups. Finally, $(T/W)$ is a finite simple $[H']$-group because the inverse image in $T$ of any nontrivial proper $[H']$-invariant normal subgroup would properly contain $W$, contradicting the maximality of $W$.

It follows that $T/ (W \cap Z)$ is a product of finite simple $[H']$-groups, so this verifies the induction step. Because the base case $n=0$ is trivial, the induction is complete.

\end{proof}

\begin{lemma}\label{finitely-many-isomorphism-classes} There exist finitely many finite simple $[H']$-groups $G_i$, up to isomorphism, such that there exists an extension $1 \to G_i \to G \to H \to 1$ of $\Gamma$-groups compatible with the actions of $H $ and $ \Gamma$ on $G_i$ by outer automorphisms, where $G$ is a level-$\mathcal C$ $\Gamma$-group. \end{lemma}

\begin{proof}This is \cite[middle of p. 22]{LWZ}. We review its short proof here:

Consider such a $G_i$.  Let $G_i'$ be a finite simple group that is a quotient of $G_i$. We must have $G_i$ isomorphic as a group to a power of $G_i'$, since otherwise the  intersection of the kernels of all surjections $G_i \to G_i'$ would be a nontrivial proper characteristic subgroup, hence a nontrivial proper $H \rtimes \Gamma$-invariant normal subgroup.

Because $G_i'$ is a Jordan-H\"{o}lder factor of a level-$\mathcal C$ $\Gamma$-group, it is a Jordan-H\"{o}lder factor of a subquotient of a product of elements of $\mathcal C$, and this it must be a Jordan-H\"{o}lder factor of an element of $\mathcal C$. Let us fix one such Jordan-H\"older factor $G_i'$, and show that powers of it give finitely many isomorphism classes. Because there are only finitely many Jordan-H\"older factors of the finitely many groups in $\mathcal C$, this means finitely many isomorphism classes overall, as desired.

First consider the case where $G_i'$ is abelian simple, thus isomorphic to $\mathbb F_p$ for some prime $p$. Then any power $G_i$ of $G_i'$ is a vector space over $\mathbb F_p$. Because $G_i$ is abelian, its outer automorphism group is its automorphism group. Thus, we can describe $G_i$ as a vector space over $\mathbb F_p$ with an action of $H'$, or in other words a $\mathbb F_p [ H']$-module. Because there are finitely many isomorphism classes of simple $\mathbb F_p [ H']$-modules, there are finitely many such $G_i$.

Next consider the case when $G_i'$ is non-abelian simple. In this case, $\Out ((G_i')^n) = \Out(G_i')^n \rtimes S_n$. Given a homomorphism from $H'$ to $\Out(G_i')^n \rtimes S_n$, if the associated $[H']$-group is simple then the image of $H'$ in $S_n$ must be transitive, because otherwise we could split $G_i$ into a product of two $[H']$-groups corresponding to two orbits. Hence $n$ is at most $|H'|$, the maximal size of a transitive $H'$-action. Because $n$ is bounded, there are only finitely homomorphisms, thus only finitely many isomorphism classes of $G_i$, as desired.
\end{proof}

\begin{defi} Let $G_1,\dots, G_m$ be pairwise non-isomorphic representatives of the finitely many isomorphism classes discussed in Lemma \ref{finitely-many-isomorphism-classes}.  \end{defi}

\begin{lemma}\label{Q-nice-description} For $G$ a finite level-$\mathcal C$ $\Gamma$-group and $\pi: G \to H$ a homomorphism, we have an isomorphism of $[H']$-groups  \[ Q( \ker \pi) \cong \prod_{i=1}^m G_i ^{e_i} \]  for some $e_1,\dots, e_m \in \mathbb N$. \end{lemma}

\begin{proof} By Lemma \ref{Q-finite-simple}, $Q(\ker \pi)$ is a product of finite simple $[H']$-groups, so it suffices by definition to show that for each factor $G'$ in the product, there exists an extension $1 \to G' \to G^*\to H \to 1$ of $\Gamma$-groups compatible with the actions of $H$ and $\Gamma$, where $G^*$ is a level-$\mathcal C$ $\Gamma$-group. To do this, observe that $Q(\ker\pi)$ is the product of $G'$ with some other groups, so $G'$ is a quotient of the $Q(\ker \pi)$ by a $H'$-invariant normal subgroup $Z$. Because $Z$ is $H' = H \rtimes \Gamma$-invariant, $Z$ remains normal and $\Gamma$-invariant as a subgroup of $G$, and $G/Z$ is the desired $G^*$. The quotient $G^*$ is level-$\mathcal C$ because the class of level-$\mathcal C$ $\Gamma$-groups is closed under quotients. \end{proof}

Now that we understand the image of the map $Q$, we can define for each measure $\mu$ on level-$\mathcal C$ $\Gamma$-groups a localized measure $\mu^H$.

\begin{defi} For $\mu$ a measure on the set of isomorphism classes of finite level-$\mathcal C$ $\Gamma$-groups, and $H$ a finite level-$\mathcal C$ $\Gamma$-group, define a measure $\mu^H$ on the set of isomorphism classes of $[H']$-groups of the form $ \prod_{i=1}^m G_i ^{e_i} $ by
\[ \mu^H (E) =  \int \left| \{ \pi : X \to H \textrm{ surjective} \mid Q( \ker \pi) \cong E \} \right| d\mu(X)  .\]
\end{defi} 

\begin{lemma}\label{mu-H-formula}  \[\mu^H(1) = \left| \Aut(H)\right|  \mu(H) \]  \end{lemma}

\begin{proof}   A surjection $X \to H$ is an isomorphism if and only if its kernel is trivial, so \[\left|\{ \pi: X \to H \textrm{ surjective} \mid  Q(\ker \pi) \cong 1 \} \right|\] is the number of isomorphisms of $X$ with $H$, and thus vanishes for $X \neq H$ and equals $|\Aut(H)|$ for $X= H$.
\end{proof}

Next, we will prove a couple lemmas to compare the moments of $\mu^H$ and the moments of $\mu$. The first will count surjections from $Q(\ker \pi)$ to $F$ using a sum over exact sequences, the second will bound the number of exact sequences, and the third will use the count of surjections to compare the moments.

\begin{defi} For $a \colon F\to G$ and $b \colon G \to H$ two maps forming an exact sequence of groups $1 \to F \to G \to H \to 1$, let $\Aut_{F,H}(G)$ denotes the group of automorphisms $\sigma$ of $G$ such that $\sigma \circ a =a $ and $b \circ \sigma =b$. \end{defi}

\begin{lemma}\label{surjection-formula} Let $X$ and $H$ be finite level-$\mathcal C$ $\Gamma$-groups. Fix $k_1,\dots, k_m\in \mathbb N$ and let $F = \prod_{i=1}^m G_i^{k_i}$. The number of pairs of a surjection $\pi\colon X\to H$ of $\Gamma$-groups and a surjection $f\colon Q(\ker \pi) \to F$ of $[H']$-groups is equal to 
\[ \sum_{ \substack{1 \to  F \to G \to H\to 1} }\frac{  \Surj_\Gamma(X, G) }{ |\Aut_{F,H}(G)|}\]  where the sum is over isomorphism classes of exact sequences of $\Gamma$-groups compatible with the actions of $H$ and $\Gamma$ on $F$ by outer automorphisms. \end{lemma} 

\begin{proof} First note that we can equivalently take $f\colon \ker \pi \to F$, because every map $f\colon \ker \pi \to F$ of $[H']$-groups factors uniquely through the projection $\ker \pi \to Q(\ker \pi)$. Indeed, the kernel of any surjection to a finite simple $[H']$-group is a maximal proper $H'$-invariant normal subgroup. Hence, the kernel of the surjection $f$ to a product of finite simple $[H']$-groups contains the intersection of all maximal proper $H'$-invariant normal subgroups. Thus, $f$ factors uniquely through the projection.

We will deduce this identity from a bijection between two sets. The first, $S_1$, is the set of pairs of a surjection $\pi:  X\to H$ of $\Gamma$-groups and a surjection $f: Q(\ker \pi) \to F$ of $[H']$-groups.

The second, $S_2$, is the set of isomorphism classes of a pair of an exact sequence $1 \to F \to G\to H \to 1$ compatible with the actions of $H$ and $\Gamma$ on $F$ by outer automorphisms and a surjection of $\Gamma$-groups $X \to G$. Equivalently, we can write $S_2$ as the disjoint union over isomorphism classes of exact sequences $1 \to F \to G \to H\to 1$ of the set of $\Gamma$-surjections modulo the action of the automorphisms $\Aut_{F,H}(G)$ of the exact sequence.  Because automorphisms of $G$ act freely on surjections $X \to G$, the size of $S_2$ is $\sum_{ \substack{1 \to  F \to G \to H\to 1} }\frac{  \Surj_\Gamma(X, G) }{ \Aut_{F,H}(G)}$, so it suffices to find the bijection.

We first construct a map from $S_2$ to $S_1$.

Given $G$ and a surjection $u\colon X\to G$, we obtain a surjection $\pi\colon X \to H$ by composition and a surjection $f\colon \ker \pi \to F$ by restricting $u$ to $\ker \pi$ and noting that its image is the kernel of the natural map $G \to H$, which is $F$. This is described by the following commutative diagram, where the square is Cartesian.
\[ \begin{tikzcd}
\ker \pi \arrow[d,"f"] \arrow[r] & X \arrow[d,"u"] \arrow[dr,"\pi"] \\
F \arrow[r] & G \arrow[r] & H \\
\end{tikzcd}\]

Because $\pi$ is a composition of two $\Gamma$-equivariant maps, it is $\Gamma$-equivariant. Because any lift of an element of $H \rtimes \Gamma$ to an automorphism of $\ker \pi$ is the action by conjugation of an element  $x \in X$ times the action of an element $\gamma \in \Gamma$, we can find a corresponding automorphism of $F$ by applying $\gamma$ and then conjugating by $u(x)$, so $f$ is an $[H']$-homomorphism.

Composing $u$ with an automorphism of $G$ fixing the inclusion of $F$ and the projection onto $H$ preserves $\pi$ and $f$.  This defines a map from $S_2$ to $S_1$.

We now find the inverse map from $S_1$ to $S_2$. 

To do this, define $G =  X/ \ker f $. Take $G \to H$ to be the projection $X/\ker f \to X/\ker \pi $. Take $F \to G$ to be the inclusion $\ker \pi/\ker f \to X /\ker f $. Take $X \to G$ the quotient map $X \to X/\ker f = G$.  These are all maps of $\Gamma$-groups, they form an exact sequence, and it is not hard to check these operations are inverses.\end{proof}

\begin{lemma}\label{extension-counting-bound} For $F = \prod_{i=1}^m G_i^{k_i}$, the number of isomorphism classes of extensions of $\Gamma$-groups $1 \to F \to G \to H \to 1$, such that the $\Gamma$-action of $F$ and the action of $H$ on $G$ by outer automorphisms are both compatible with the $H \rtimes \Gamma$-structure of $F$, is $O \left( O(1)^{\sum_{i=1}^m k_i}\right) $, where the constants depend only on $\Gamma$ and $H$. \end{lemma}

\begin{proof} An extension of $\Gamma$-groups $1 \to F \to G \to H\to 1$ is equivalent to an extension of groups $1 \to F \to G \rtimes \Gamma \to H \rtimes \Gamma\to 1$ where $G$ maps to $H$ and $\Gamma$ maps to $\Gamma$. To classify such extensions, fix for each element $h\in H \rtimes \Gamma$ a lift $\sigma_h$ of the associated outer automorphism of $F$ to an automorphism of $F$.

Then, given such an extension, choose for each $h\in H \rtimes \Gamma $ an element $\alpha_h \in G\rtimes \Gamma$ whose image in $H \rtimes \Gamma$ is $h$ and whose action by conjugation on $F$ is by $\sigma_h$. This is possible as we can adjust the conjugation action by an inner automorphism by multiplying by an appropriate element of $F$.

Because projection to $H \rtimes \Gamma $ is a group homomorphism with kernel $F$, there exists for each $h_1,h_2 \in H \rtimes \Gamma$ an element  $f_{h_1,h_2} \in F$ such that $\alpha_{h_1} \alpha_{h_2}= f_{h_1,h_2} \alpha_{h_1h_2}$. 

We can express each element of $G\rtimes \Gamma$ as $f \alpha_h$ for some $f\in F$, and we have 
\[ f_1 \alpha_{h_1} f_2 \alpha_{h_2} = f_1 (\alpha_{h_1} f_2 \alpha_{h_1}^{-1}) \alpha_{h_1} \alpha_{h_2} = f_1 \sigma_{h_1}(f_2)  f_{h_1 ,h_2} \alpha_{h_1h_2} \]
so the multiplication table of $G \rtimes \Gamma$, expressed this way, is determined by the elements $f_{h_1,h_2}$.  To describe the group with this multiplcation table, which comes equipped with a projection to $\Gamma$, as a semidirect product $G \rtimes \Gamma$, we need to fix a subgroup that maps isomorphically $\Gamma$ under the projection $H \rtimes \Gamma$. This requires fixing a lift of each element of $\Gamma$, which represents at most $|F|^{\Gamma}$ additional choices. 

Thus, the number of possible isomorphism classes of exact sequences is at most \[|F|^{ ( |H \rtimes \Gamma|^2 + |\Gamma| ) } = \prod_{i=1}^m |G_i|^{ k_i ( |H \rtimes \Gamma|^2 + |\Gamma| ) } = O(1)^{ \sum_{i=1}^m k_i}.\] \end{proof}

We can likely express the data $\alpha_{h_1h_2}$ in this proof using group cohomology to get a more precise count, but this isn't necessary for the bound.

\begin{lemma}\label{surjection-integral-formula} For $\mu$ a measure on the set of isomorphism classes of finite level-$\mathcal C$ $\Gamma$-groups, $H$ a finite level-$\mathcal C$ $\Gamma$-group, and $F = \prod_{i=1}^m G_i^{k_i}$, we have
\[ \int \Surj_{[H']}  (E, F)  d\mu^H (E) =  \sum_{ 1 \to F \to G \to H \to 1} \int  \Surj_{\Gamma} (X, G)    d\mu(X) / \Aut_{F,H}(G).   \] \end{lemma}

\begin{proof} We have \[ \int \Surj_{[H']}  (E, F)  d\mu^H (E) = \int  \sum_ {  \pi \colon X \to H \textrm{ surjective}}  \Surj_{[H']} ( Q(\ker \pi), F)   d\mu(X)  \] 
\[ = \int   \sum_{ 1 \to F \to G \to H \to 1}  \frac{  \Surj_{\Gamma} (X, G) }{  \Aut_{F,H}(G)}  d\mu(X)   = \sum_{ 1 \to F \to G \to H \to 1} \int  \Surj_{\Gamma} (X, G)    d\mu(X) / \Aut_{F,H}(G). \] where the first identity is by definition, the second is Lemma \ref{surjection-formula}, and the third exchanges the integral with a sum which, by Lemma \ref{extension-counting-bound}, is finite. 
\end{proof}

Now we are finally ready to prove the main theorem:

\begin{theorem}\label{main-bound-group} Let $\Gamma$ be a finite group. Let $\mathcal C$ be a finite set of finite $\Gamma$-groups. 
Let $\mu$ be a measure on the set of isomorphism classes of finite level-$\mathcal C$ $\Gamma$-groups. Let $\mu_t$ be a sequence of measures on the same set. Assume that, for $H$ a finite level-$\mathcal C$ $\Gamma$-group, we have
\begin{equation}\label{assumption-limit} \lim_{t \to \infty} \int \Surj_\Gamma(X, H) d\mu_t(X)  = \int \Surj_\Gamma(X, H) d\mu(X) .\end{equation}
and
\begin{equation}\label{assumption-bound} \int \Surj_\Gamma(X, H) d\mu(X)  = O (   |H|^{O(1) }) \end{equation}
Then for all finite level-$\mathcal C$ $\Gamma$-groups $H$ we have \begin{equation}\label{groups-conclusion} \lim_{t \to \infty} \mu_t(H) = \mu(H) .\end{equation}

\end{theorem}

\begin{proof}

First, we will check that the hypotheses of Lemma \ref{flat-case} apply to the measures $\mu^H$ and $\mu^H_t$.

Let $F = \prod_{i=1}^m G_i^{k_i}$.

By applying Lemma \ref{surjection-integral-formula} to $\mu$ and $\mu_t$, we have 
\[\lim_{t \to \infty}  \int \Surj_{[H']}  (E, F) d \mu_t^H (E)  = \lim_{t\to \infty} \sum_{ 1 \to F \to G \to H \to 1} \int  \Surj_\Gamma (X, G)    d\mu_t(X) / \Aut_{F,H}(G)\]\[ =  \sum_{ 1 \to F \to G \to H \to 1}  \lim_{t\to \infty} \int  \Surj_\Gamma (X, G)    d\mu_t(X) / \Aut_{F,H}(G)\]\[=  \sum_{ 1 \to F \to G \to H \to 1} \int  \Surj_\Gamma(X, G)    d\mu(X) / \Aut_{F,H}(G)  =  \int \Surj_{[H']} (E, F) d \mu^H (E) \] where we use the finiteness of the sum, from Lemma \ref{extension-counting-bound}, to exchange it with a limit and also our assumption \eqref{assumption-limit} to calculate the limit. This verifies the assumption  \eqref{flat-limit} of Lemma \ref{flat-case}.

By Lemma \ref{surjection-integral-formula} and assumption \eqref{assumption-bound}  we have\[ \int \Surj_{[H']}  (E, F)  d\mu^H (E) =  \sum_{ 1 \to F \to G \to H \to 1} \int  \Surj_{\Gamma} (X, G)    d\mu(X) / \Aut_{F,H}(G)\] \[= \sum_{ 1 \to F \to G \to H \to 1}  O ( |G|^{O(1)} ).   \] 

Because $|G| = |F| |H| =|H| \prod_{i=1}^{m} |G_i|^{k_i} = O \left( O(1)^{\sum_{i=1}^m k_i}\right)$, and by Lemma \ref{extension-counting-bound} the number of terms is $O \left( O(1)^{\sum_{i=1}^m k_i}\right)$, the total sum is $O \left( O(1)^{\sum_{i=1}^m k_i}\right)$, giving the assumption \eqref{flat-bound} of \eqref{flat-case}.

So we may apply Lemma \ref{flat-case}, obtaining \[ \lim_{t \to \infty} \mu^H_t(1) = \mu^H(1), \] which by Lemma \ref{mu-H-formula}, applied to both $\mu$ and $\mu_t$, gives our desired \eqref{groups-conclusion}.  \end{proof}

\begin{proof}[Proof of Theorem \ref{main-function-field-theorem}]
Fix $\Gamma$ and $\mathcal C$, and for $G$ any level-$\mathcal C$ $\Gamma$-group, let    \[ \mu_{q, n} ( X) =  \frac{ \sum_{ d=0}^{n}   \left|\left \{ K \in E_{\Gamma}(d,q)\mid  \Gal(K^{\#} /K)^{\mathcal C} \cong X \right\} \right| } { \sum_{ d=0}^{n} \left|E_{\Gamma}(d,q) \right| } \] 
and \[ \mu(X) =  \mu_{\Gamma} (U_{\mathcal C, X})  \]
for $\mu_{\Gamma}$ the measure defined in \cite{LWZ} and $U_{\mathcal C, X}$ is the set of all profinite $\Gamma$-groups whose pro-$\mathcal C$ completion is $\Gamma$-isomorphic to $X$.

We have
\begin{equation}\label{ff-moment-unfolding}\begin{split}  \int \Surj_\Gamma (X, H) d\mu_{q, n} (X)   = \sum_{X \textrm{ level-}\mathcal C\textrm{ }\Gamma\textrm{-group}}  \frac{ \sum_{ d=0}^{n}   \left|\left \{ K \in E_{\Gamma}(d,q)\mid  \Gal(K^{\#} /K)^{\mathcal C} \cong X \right\} \right| \Surj_\Gamma(X, H)  } { \sum_{ d=0}^{n} \left|E_{\Gamma}(d,q)\right| } \\
 =  \frac{ \sum_{ d=0}^{n}   \sum_{\substack{ K \in E_{\Gamma}(d,q) }} \Surj_\Gamma(\Gal(K^{\#} /K)^{\mathcal C} , H)  } { \sum_{ d=0}^{n} \left|E_{\Gamma}(d,q)\right| } =  \frac{ \sum_{ d =0}^{n}   \sum_{\substack{ K \in E_{\Gamma}(d,q) }} \Surj_\Gamma(\Gal(K^{\#} /K) , H)  } { \sum_{ d=0}^{n} \left|E_{\Gamma}(d,q)\right| } \end{split}\end{equation}
so 
\begin{equation}\label{ff-case-limit} \begin{split} &\lim_{n \to \infty} \lim_{\substack{ q \to \infty \\ \gcd(q, |\Gamma| |\mathcal C|)=1 \\ \gcd(q-1, |\mathcal C|)=1 }}  \int \Surj_\Gamma(X, H) d\mu_{q, n} (X) \\
=&\lim_{n \to \infty} \lim_{\substack{ q \to \infty \\ \gcd(q, |\Gamma| |\mathcal C|)=1 \\ \gcd(q-1, |\mathcal C|)=1 }}  \frac{ \sum_{d=0}^{n}   \sum_{\substack{ K \in E_{\Gamma}(d,q) }} \Surj_\Gamma(\Gal(K^{\#} /K) , H)  } { \sum_{ d=0}^{n} \left|E_{\Gamma}(d,q)\right| } 
  =\int \Surj_{\Gamma}(G, H)  d\mu_{\Gamma}(G) \\ & = \int \Surj_\Gamma (G^{\mathcal C}, H) d \mu_{\Gamma }(G) = \int \Surj_{\Gamma} (X, H) d\mu (X) , \end{split} \end{equation} where the first identity is \eqref{ff-moment-unfolding}, the second identity is \cite[Theorem 1.4]{LWZ}, the third identity is because $H$ is level-$\mathcal C$, and the last identity follows from the definition of $\mu$ in terms of $\mu_{\Gamma}$.

Furthermore, we have by \cite[Theorem 6.2]{LWZ}
\begin{equation}\label{ff-case-bound}  \int \Surj_{\Gamma} (X, H) d\mu (X)  = \int \Surj_{\Gamma}(G, H)  d\mu_{\Gamma}(G) = \frac{1}{[H : H^\Gamma] }=  O ( |H|^{O(1)} ). \end{equation}

Hence taking any sequence of $n$ and $q$ going to $\infty$, with $q$ satisfying the congruence conditions and growing sufficiently fast with respect to $n$, \eqref{ff-case-limit} and \eqref{ff-case-bound} verify the assumptions \eqref{assumption-limit} and \eqref{assumption-bound} of Theorem \ref{main-bound-group}. Hence we can apply Theorem \ref{main-bound-group}, and its conclusion implies our desired statement by the following Lemma~\ref{convergence-step}.\end{proof} 

\begin{lemma}\label{convergence-step} Let $F(n,q)$ be a function of variables $n,q$ where $n$ is a natural number and $q$ is restricted to some infinite subset of natural numbers. Assume, that, for all sequences $(n_i,q_i)$ of pairs with $q_i$ growing sufficiently fast with respect to $n_i$, $\lim_{i\to\infty} F(n_i, q_i ) =x$. Then \[ \limsup_{n\to\infty} \limsup_{q\to\infty} = \liminf_{n\to\infty} \liminf_{q\to\infty} F(n,q) =x.\]
\end{lemma}

\begin{proof} We have \[ \limsup_{n\to\infty} \limsup_{q\to\infty}  \geq \limsup_{n\to\infty} \liminf_{q\to\infty} F(n,q) \geq  \liminf_{n\to\infty} \liminf_{q\to\infty} F(n,q)\] so it suffices to prove \[  \liminf_{n\to\infty} \liminf_{q\to\infty} F(n,q) \geq x \geq  \limsup_{n\to\infty} \limsup_{q\to\infty}F(n,q)  \] and since those two equalities are exchanged by negating $F$ and $x$ it suffices to prove $x \geq  \limsup_{n\to\infty} \limsup_{q\to\infty}$. Suppose for contradiction that $x <  \limsup_{n\to\infty} \limsup_{q\to\infty}F(n,q) $. Then $ \limsup_{n\to\infty} \limsup_{q\to\infty}F(n,q) \geq  x+ \epsilon$ for some $\epsilon>0$. Thus for infinitely many $n$ we have $ \limsup_{q\to\infty} F(n,q)  \geq   x+ \epsilon/2$ and then for each of those $n$, for infinitely many $q$ we have $ F(n,q)  \geq    x+ \epsilon/4$.  We can thus take a sequence $n_1,n_2,\dots$ such that for each $i$ we have $F(n_i,q) \geq x+\epsilon/4$ for infinitely many $q$, and then choose $q_i$ among those $q$ growing arbitrarily large as a function of $n_i$, so we have $F(n_i,q_i) \geq x+ \epsilon/4$ for all $i$, contradicting $\lim_{i\to\infty} F(n_i, q_i ) =x$.\end{proof}

\section{Modules}

In this section, we assume that $R$ is an associative algebra such that there are finitely many finite simple $R$-modules, and we let $M_1,\dots M_m$ be representatives of these isomorphism classes, pairwise non-isomorphic.

 Let $I$ be the intersection in $R$ of the annihilators of the finite simple $R$-modules. Because $R/I$ is a product of finite simple algebras, every $R/I$-module is a product of finite simple $R$-modules, and so has the form $\prod_{i=1}^m M_i^{e_i}$ for $e_1,\dots e_m \in \mathbb N$.

\begin{lemma}\label{module-flat-case}
 Let $\tilde{\mu}$ be a measure on the set of isomorphism classes of modules $\prod_{i=1}^m M_i^{e_i}$ and let $\tilde{\mu}_t$ be a sequence of such measures. 

Assume that for all $k_1,\dots, k_m \in \mathbb N$ we have
\begin{equation}\label{module-flat-bound} \sum_{e_{1}, \dots ,e_m=0}^{\infty}  \Surj_{R} \left( \prod_{i=1}^m M_i^{e_i} , \prod_{i=1}^m M_i^{k_i}\right)  \tilde{\mu} \left(  \prod_{i=1}^m M_i^{e_i} \right)  = O \left( O(1)^{ \sum_{i=1}^m k_i } \right) \end{equation}
and
\begin{equation}\label{module-flat-limit}   \lim_{t \to \infty}  \sum_{e_{1}, \dots ,e_m=0}^{\infty}  \Surj_{R} \left( \prod_{i=1}^m M_i^{e_i} , \prod_{i=1}^m M_i^{k_i}\right)  \tilde{\mu}_t \left(  \prod_{i=1}^m M_i^{e_i} \right)  \end{equation}\[ =   \sum_{e_{1}, \dots ,e_m=0}^{\infty}  \Surj_{R} \left( \prod_{i=1}^m M_i^{e_i} , \prod_{i=1}^m M_i^{k_i}\right)  \tilde{\mu} \left(  \prod_{i=1}^m M_i^{e_i} \right).\]
Then for all $k_{j+1},\dots, k_m \in \mathbb N$ we have
\begin{equation}\label{module-flat-conc}  \lim_{t \to \infty} \tilde{\mu}_t ( 0 ) = \tilde{\mu} ( 0) . \end{equation}\end{lemma}

\begin{proof} The proof is almost word-for-word identical to the combination of the proofs of Lemmas \ref{c-k-exists}, \ref{one-variable}, \ref{inductive-step}, and \ref{flat-case}, except that we replace every occurrence of ``finite simple $[H']$-group" in those lemmas with ``finite simple $R$-module", we replace $\Surj_{[H']}$ with $\Surj_R$, and in Lemma \ref{c-k-exists} we ignore the non-abelian case. (The analogues of Lemmas \ref{abelian-surjection-count} and \ref{surjection-count-splitting} are immediate in this setting.) Hence we do not repeat the proof. \end{proof}

\begin{defi} For $M$ a finite $R$-module, and $\mu$ a measure on the set of isomorphism classes of finite $R$-modules, we define a measure $\mu^M$ on the set of $R/I$-modules $N$ by
\[ \mu^M(N) = \int \left|  \{ \pi\colon X \to M \textrm{ surjective} \mid (\ker \pi)/I \cong N \} \right| d\mu(X) .\] \end{defi} 

\begin{lemma}\label{module-measure-zero}  For $\mu$ a measure on the set of isomorphism classes of finite $R$-modules, we have
\[ \mu^M(0)= | \Aut_R(M)| \mu(H).\] 
\end{lemma}  \begin{proof} A surjection $X \to M$ is an isomorphism if and only if its kernel is zero, so \ \[\left|\{ \pi\colon X \to M \textrm{ surjective} \mid  (\ker \pi)/I \cong 0 \} \right|\] is the number of isomorphisms of $X$ with $M$, and thus vanishes for $X \neq M$ and equals $|\Aut_R(M)|$ for $X= M$. \end{proof}

\begin{lemma}\label{module-surjection-formula} Let $X$ and $M$ be finite $R$-modules and fix $k_1,\dots, k_m \in \mathbb N$. Let $N = \prod_{i=1}^m M_i^{k_i}$. The number of pairs of a surjection $\pi\colon  X\to M$ and a surjection $f\colon \ker \pi/I \to N$ is equal to 
\[ \sum_{ \substack{0 \to  N \to M' \to M\to 0} }\frac{  \Surj_R(X, M') }{ | \Hom_R (M, N )| }\]  where the sum is over isomorphism classes of exact sequences of $R$-modules. \end{lemma} 

\begin{proof}  Because $N$ is an $R/I$-module, surjections $f\colon \ker \pi/I \to N$ are in bijection with surjections $f\colon \ker \pi \to N$, which we will use for the remainder of the proof.

We check that there is a bijection between (isomorphism classes of) pairs of a surjection $\pi\colon X\to M$ and  $f\colon \ker \pi \to N $ and (isomorphism classes of) pairs of an exact sequence $0 \to  N \to M' \to M\to 0$ with a surjection $X\to M'$. We will then deduce the counting formula.

Given surjections $\pi\colon X\to M$ and  $f\colon \ker \pi \to N $, we can define $M' = X/ (\ker f)$. The filtration $(\ker \pi) \subseteq (\ker f) \subseteq X$ gives an exact sequence $0 \to N \to M' \to M \to 0$ where we use $M = X/(\ker \pi)$ and $N= (\ker \pi)/(\ker f)$.  We then have a surjection $X \to M' $ because $M'$ is a quotient of $X$.  

Conversely, given an exact sequence $0 \to  N \to M' \to M\to 0$ and a surjection $X \to M'$, the composition $X \to M' \to M$ is a surjection, and its kernel is $X \times_M 0 = X\times_{M'} (M' \times_M 0) = X \times_{M'} N $ and thus surjects onto $N$. 

It is not too hard to check that these operations are inverse.

Finally, note that the number of isomorphism classes of pairs of an exact sequence $0 \to  N \to M' \to M\to 0$ with a surjection $X\to M'$ is equal to the sum over isomorphism classes of exact sequences of the number of surjections $X \to M'$, up to automorphisms of that exact sequence. The automorphisms of the exact sequence are $\Hom_R(M, N)$, and they act freely on surjections, so we simply divide the count by $|\Hom_R(M, N)|$. \end{proof}

\begin{lemma}\label{module-surjection-integral-formula} For $\mu$ a measure on the set of isomorphism classes of finite $R$-modules, $M$ a finite $R$-module, and $N = \prod_{i=1}^m M_i^{k_i}$, we have
\[ \int \Surj_{R}  (S,N )  d\mu^M (S) =  \sum_{ 0 \to N \to M' \to M \to 0} \int  \Surj_{R} (X, M')    d\mu(X) / |\Hom_R (M,N)| .   \] \end{lemma}

\begin{proof} We have \[ \int \Surj_{R}  (S,N )  d\mu^M (S)= \int  \sum_ {  \pi \colon X \to M \textrm{ surjective}}  \Surj_R( (\ker \pi)/I, N )   d\mu(X)  \] 
\[ = \int   \sum_{ 0 \to N \to M' \to M\to 0 }  \frac{  \Surj_{R} (X, M') }{  |\Hom_R(M,N)| }  d\mu(X)   = \sum_{ 0 \to N \to M' \to M\to 0 }    \int  \Surj_{R} (X, M')    d\mu(X) /|\Hom_R(M,N)| \] where the first identity is by definition, the second is Lemma \ref{module-surjection-formula}, and the third exchanges the integral with a sum of nonnegative functions. 
\end{proof}

 \begin{theorem}\label{main-bound-algebra} Let $R$ be an associative algebra. Assume that there are finitely many isomorphism classes of finite simple $R$-modules and that $\operatorname{Ext}^1$-groups between finite $R$-modules are finite.
 
 Let $\mu$ be a measure on the set of isomorphism classes of finite $R$-modules. Let $\mu_t$ be a sequence of measures on the same set. Assume that, for $M$ a finite $R$-module,
 \begin{equation}\label{algebra-limit}\lim_{t \to \infty} \int \Surj_R (X, M) d\mu_t(X)  = \int \Surj_R(X, M) d\mu(X) .\end{equation}
and
\begin{equation}\label{algebra-bound} \int \Surj_R(X, M) d\mu(X)  = O (   |M|^{O(1) }) \end{equation}
Then for all finite $R$-modules $M$ we have
\begin{equation}\label{algebra-conclusion} \lim_{t \to \infty} \mu_t(M) = \mu(M) .\end{equation}

\end{theorem}

\begin{proof}

Let us check that the hypotheses of Lemma \ref{module-flat-case} apply to the measures $\mu^M$ and $\mu^M_t$.

Let $N = \prod_{i=1}^n M_i^{k_i}$ for some tuple $k_1,\dots, k_n$ of nonnegative integers. By applying Lemma \ref{module-surjection-integral-formula} to $\mu$ and $\mu_t$, we have 
\[\lim_{t \to \infty}  \int \Surj_{R}  (S, N)  d\mu_t^M (S)  = \lim_{t\to \infty} \sum_{ 0 \to N \to M'\to M \to 0} \int  \Surj_R (X, M')    d\mu_t(X) /|\Hom_R(M,N)| \]\[ =  \sum_{ 0 \to N \to M'\to M \to 0} \lim_{t\to \infty} \int   \Surj_R (X, M')    d\mu_t(X) /|\Hom_R(M,N)| \]\[=  \sum_{ 0 \to N \to M'\to M \to 0}  \int \Surj_R (X, M')    d\mu(X) /|\Hom_R(M,N) |= \int \Surj_{R}  (S, N)  d\mu^M (S) \] where we use the assumed finiteness of $\operatorname{Ext}^1$ to exchange the sum with a limit and also the assumed \eqref{algebra-limit} to calculate the limit. This verifies the assumption  \eqref{module-flat-limit} of Lemma \ref{module-flat-case}.

By Lemma \ref{module-surjection-integral-formula} and assumption \eqref{algebra-bound}  we have\[\int \Surj_{R}  (S, N) d \mu^M (S) =   \sum_{ 0 \to N \to M'\to M \to 0}  \int \Surj_R (X, M')    d\mu(X) /|\Hom_R(M,N) \] \begin{equation}\label{module-sum-to-bound}=  \sum_{ 0 \to N \to M'\to M \to 0}  O (|M'|^{O(1)}) /|\Hom_R(M,N)| .   \end{equation}

For $N = \prod_{i=1}^m M_i^{k_i}$,  \[|M'| = |M| |N| =|M| \prod_{i=1}^{m} |M_i|^{k_i} = O \left( O(1)^{\sum_{i=1}^m k_i}\right),\] and the number of terms is \[| \Ext^1( M,N) | = \prod_{i=1}^m |\Ext^1 (M, M_i)|^{k_1} = O \left( O(1)^{\sum_{i=1}^m k_i }\right),\] so the total sum \eqref{module-sum-to-bound} is $O \left( O(1)^{\sum_{i=1}^m k_i}\right)$, giving the assumption \eqref{module-flat-bound} of \eqref{module-flat-case}.

So we may apply Lemma \ref{module-flat-case}, obtaining \[ \lim_{t \to \infty} \mu^M_t(0) = \mu^M(0), \] which by Lemma \ref{module-measure-zero}, applied to both $\mu$ and $\mu_t$, gives our desired \eqref{algebra-conclusion}.  \end{proof}


\begin{thebibliography}{9}

\bibitem{BW} Nigel Boston and Melanie Matchett Wood, Nonabelian Cohen-Lenstra Heuristics over Function Fields, \href{https://arxiv.org/abs/1604.03433}{arXiv:1604.03433}.

 \bibitem{CL} H.~Cohen and H.~W.~Lenstra, Jr. Heuristics on class groups of number fields. In {\em Number Theory, Noordwijkerhout 1983}, Lecture Notes in Mathematics {\bf 1068} (1984), 33-62.

\bibitem{CM} H.~Cohen and J.~Martinet. Class groups of number fields: numerical heuristics. {\em Mathematics of Computation}, {\bf 48} (1987), 123-137.

\bibitem{EVW} Jordan S.~Ellenberg, Akshay Venkatesh, and Craig Westerland. Homological stability for Hurwitz spaces and the Cohen-Lenstra conjecture over function fields. {\em Annals of Mathematics}, {\bf 183} (2016), 729-786.

\bibitem{HB-o} D.~R.~Heath-Brown, The size of the Selmer groups for the congruent number problem, II, \href{http://citeseerx.ist.psu.edu/viewdoc/download?doi=10.1.1.538.9147&rep=rep1&type=pdf}{preprint version}

\bibitem{HB-i} D.~R.~Heath-Brown, The size of the Selmer groups for the congruent number problem, II, {\em Inventiones Mathematicae} {\bf 118} (1994), 331-370.

\bibitem{LWZ}
Yuan Liu, Melanie Matchett Wood, and David Zurieck-Brown, A predicted distribution for Galois groups of maximal unramified extensions, \href{https://arxiv.org/abs/1907.05002}{arXiv:1907.05002}

\bibitem{WW}
Weitong Wang and Melanie Matchett Wood, Moments and interpretations of the Cohen-Lenstra-Martinet heuristics, \href{https://arxiv.org/abs/1907.11201}{arXiv:1907.11201}.

\end{thebibliography}
\end{document}